\documentclass[11pt]{amsart}
\usepackage{amsmath,amssymb,amsthm,latexsym,cite,cancel}
\usepackage[small]{caption}
\usepackage{graphicx,wasysym,overpic,tikz,color}
\usepackage{subfigure,color}
\usepackage{cite}
\usepackage[colorlinks=true,urlcolor=blue,
citecolor=red,linkcolor=blue,linktocpage,pdfpagelabels,
bookmarksnumbered,bookmarksopen]{hyperref}
\usepackage[italian,english]{babel}
\usepackage{units}
\usepackage{enumitem}
\usepackage[left=2.1cm,right=2.1cm,top=2.71cm,bottom=2.71cm]{geometry}
\usepackage[hyperpageref]{backref}

\usepackage[colorinlistoftodos]{todonotes}
\newtheorem{theorem}{Theorem}[section]
\newtheorem{definition}[theorem]{Definition}
\newtheorem{proposition}[theorem]{Proposition}
\newtheorem{lemma}[theorem]{Lemma}
\newtheorem{remark}[theorem]{Remark}
\newtheorem{example}[theorem]{Example}

\newcommand{\abs}[1]{\lvert#1\rvert}
\newcommand{\norm}[1]{\lVert#1\rVert}
\newcommand{\red}[1]{\textcolor{red}{#1}}

\numberwithin{equation}{section}

\title[A logarithmic Schrodinger equation]{Existence and multiplicity of solutions for the logarithmic Schr\"{o}dinger equation with a potential on lattice graphs}

\author[Z.T. He]{Zhentao He}

\address[Z.T. He]{\newline\indent
	School of Mathematics
	\newline\indent
	East China University of Science and Technology
	\newline\indent
	Shanghai 200237, PR China }
\email{\href{mailto:hezhentao2001@outlook.com}{hezhentao2001@outlook.com}}

\author[C. Ji]{Chao Ji}

\address[C. Ji]{\newline\indent
	School of Mathematics
	\newline\indent
	East China University of Science and Technology
	\newline\indent
	Shanghai 200237, PR China }
\email{\href{mailto:jichao@ecust.edu.cn}{jichao@ecust.edu.cn}}

\date{\today}

\begin{document}

\begin{abstract}
In this paper, we consider the existence and multiplicity of solutions for the logarithmic Schr\"{o}dinger equation on lattice graphs $\mathbb{Z}^N$
$$
-\Delta u+V(x) u=u \log u^2, \quad x \in \mathbb{Z}^N,
$$
When the potential $V$ is coercive, we obtain infinitely many solutions by adapting some arguments of the Fountain theorem. In the cases of periodic potential, asymptotically periodic potential and bounded potential,  we first investigate the existence of ground state solutions via the variation methods, and then we generalize these results from $\mathbb{Z}^N$ to quasi-transitive graphs. Finally, we extend the main results of the paper to the $p$-Laplacian equation with the logarithmic nonlinearity.
\end{abstract}

\maketitle

\begin{center}
	\begin{minipage}{11cm}
		\tableofcontents
	\end{minipage}
\end{center}

\section{Introduction and main results}
In the past decade, a vast literature is dedicated to the following nonlinear logarithmic  Schr\"{o}dinger equation
\begin{equation}\label{loga1}
-\Delta u+V(x) u=u \log u^2, \,\,\text { in } \mathbb{R}^N.
\end{equation}
Because this class of equations has some important physical applications, such as quantum mechanics, quantum optics, nuclear physics, transport and diffusion phenomena, open quantum systems, effective quantum gravity, theory of superfluidity and Bose-Einstein condensation (see \cite{Zl} and the references therein). On the other hand, the logarithmic Schr\"{o}dinger equation also raises many difficult mathematical problems, for example, the associated energy functional is not of class $C^{1}$ in $H^1\left(\mathbb{R}^N\right)$. In fact, it is not continuous in $H^1\left(\mathbb{R}^N\right)$, because there exists $u \in H^1\left(\mathbb{R}^N\right)$ such that $\int_{\mathbb{R}^N} u^2 \log u^2 d x=-\infty$. In order to overcome this technical difficulty, some authors have used different techniques to study the existence, multiplicity and concentration of the solutions under some assumptions on the potential $V$, which can be seen in \cite{AD, AJ1, AJ3, AJ4, DMS, JS, SS, SS1, WZ, ZW1, ZZ} and the references therein.

In recent years, the various mathematical problems on graphs have been extensively investigated (see \cite{CWY, Gr, Gri1, hua2021existence, HX1, HLY, Shao}  and references therein).  In particular, in the monograph \cite{Gr}, Grigor'yan introduced the discrete Laplace operator on finite and infinite graphs; in \cite{HLY}, Huang, Lin and Yau proved two existence results of solutions to mean field equations on an arbitrary connected finite graph; in \cite{HX1}, applying the Nehari method, Hua and Xu studied the existence of ground states of the nonlinear Schr\"{o}dinger equation $-\Delta u+V(x) u=f(x, u)$ on the lattice graph $\mathbb{Z}^N$ where $f$ satisfies some growth conditions and the potential function $V$ is periodic or bounded. Due to our
scope, we would like to mention the recent contribution \cite{CWY} where Chang et al. studied the following logarithmic Schr\"{o}dinger equations on locally finite graphs
\begin{equation}\label{eqlogv}
    -\Delta u+V(x) u=u \log u^2 \,\, \text { in } \,\mathbb{V},
\end{equation}
where $\Delta$ is the graph Laplacian, $G=(\mathbb{V}, E)$ is a connected locally finite graph, the potential $V: \mathbb{V} \rightarrow \mathbb{R}$ is bounded from below and may change sign. We know that one of the significant challenges in studying the above problem is the lack of compactness. Under two appropriate assumptions on potential $V(x)$, the authors first established two Sobolev compact embedding theorems. Then, they obtained  ground state solutions of the above problem by using the Nehari manifold method and the mountain pass theorem respectively. However, under the general assumptions on the potential $V(x)$, such Sobolev compact embedding results are not to be expected.

Motivated by \cite{SS, JS, HX1}, in the present paper,  we will study the logarithmic Schr\"{o}dinger equation with the potentials on lattice graph
\begin{equation}\label{loga}
-\Delta u+V(x) u=u \log u^2, \quad x \in \mathbb{Z}^N.
\end{equation}
Specifically, when the potential is coercive, we establish the existence of infinitely many solutions by adapting some arguments of the Fountain theorem and a Sobolev compact embedding theorem, and in the case of periodic potential, asymptotic periodic potential and  bounded potential, we obtain the existence of the ground state solutions via variational methods. Here it is essential to note that, in the case of periodic potential, asymptotic periodic potential and  bounded potential,  the compact embedding theorems in \cite{CWY} are absent, even on lattice graph. Therefore, our problem is more challenging and complicated. Another novelty arises when the potential $V(x)$ is bounded and $\underset{x\in \mathbb{Z}^N}\inf V(x)> -1$, in contrast to the continuous case, due to the equivalence of two spaces $H^1\left(\mathbb{V}\right)$ and $\ell^2\left(\mathbb{V}\right)$(see the detail in Section 2), there is no need to rely on logarithmic Sobolev inequality to prove the boundedness of the Palais-Smale ($(PS)$, for short) sequence. In the case of a coercive potential, the previous arguments dees not work,  we establish logarithmic Sobolev inequality on lattice graph to show the boundedness of the $(PS)$ sequence. We believe that this has independent interest and can be applied to the research of other related problems.

Now we recall some basic setting on graphs from \cite{HX1, CWY}. Let $G=(\mathbb{V}, \mathbb{E})$ be a connected, undirected, locally finite graph, where $\mathbb{V}$ denotes a set of vertices and $\mathbb{E}$ denotes a set of edges. Let us write $x \sim y(x$ is a neighbour of $y)$ if $xy \in \mathbb{E}$.  A graph is called locally finite if each vertex has finitely many neighbours. A graph is connected if any two vertices $x$ and $y$ can be connected via finite edges.   The graph $G$ is undirected means that each $xy\in \mathbb{E}$ is unordered. For any edge $x y \in \mathbb{E}$, we assume the weight $\omega_{x y}$ satisfies $\omega_{x y}=\omega_{y x}>0$. The degree of $x \in \mathbb{V}$ is defined as $\operatorname{deg}(x)=\sum_{y \sim x} \omega_{x y}$. The distance $d(x, y)$ of two vertices $x, y \in \mathbb{V}$ is defined by the minimal number of edges which connect $x$ and $y$. If the distance $d(x,y)$ is uniformly bounded from above for any $x, y \in \Omega$, we call $\Omega$ a bounded domain in $\mathbb{V}$. The boundary of $\Omega$ is defined as $\partial\Omega := \{y\not\in\Omega : \exists x\in\Omega \text{ such that } xy \in \mathbb{E}\}$. The measure $\mu: \mathbb{V} \rightarrow \mathbb{R}^{+}$on the graph is a finite positive function on $G$. Denote the set of functions on $G$ by $C(\mathbb{V})$, for any function $u\in C(\mathbb{V})$, the graph Laplacian of $u$ is defined by
$$
\Delta u(x):=\frac{1}{\mu(x)}\sum_{y \in \mathbb{V}, y \sim x}w_{xy}(u(y)-u(x)) .
$$

Throughout this paper, we consider the case $w_{xy}=1$ for any $xy\in \mathbb{E}$ and  $\mu$ be the counting measure on $\mathbb{V}$, i.e., for any subset $A \subset \mathbb{V}, \mu(A):=\#\{x: x \in A\}$. For any function $f$ on $\mathbb{V}$, we write
$$
\int_{\mathbb{V}} f d \mu:=\sum_{x \in \mathbb{V}} f(x)
$$
whenever it makes sense. Under these assumptions, for any function $u\in C(\mathbb{V})$, the graph Laplacian of $u$ is defined by
$$
\Delta u(x):=\sum_{y \in \mathbb{V}, y \sim x}(u(y)-u(x)).
$$
Let  $C_c(\mathbb{V})$ be the set of all functions with finite support and  $H^{1}(\mathbb{V})$ is the completion of $C_c(\mathbb{V})$ under the norm
$$
\|u\|:=\left(\frac{1}{2} \sum_{x \in \mathbb{V}} \sum_{y \sim x}(u(y)-u(x))^2+\sum_{x \in \mathbb{V}} u^2(x)\right)^{1 / 2} .
$$
For any $p \in[1, \infty]$, let $\ell^p(\mathbb{V})$ be the space of $\ell^p$ summable functions on $\mathbb{V}$ and write $\|\cdot\|_p$ as the $\ell^p(\mathbb{V})$ norm, i.e.,
$$
\|u\|_p:=\left\{\begin{array}{lc}
\Big(\sum_{x \in \mathbb{V}}|u(x)|^p\Big)^{1/p}, & 1 \leq p<\infty, \\
\sup _{x \in \mathbb{V}}|u(x)|, & p=\infty .
\end{array}\right.
$$
Let us denote by $\mathbb{Z}^N$ the standard lattice graph with the set of vertices
$$
\left\{x=\left(x_1, \ldots, x_N\right): x_i \in \mathbb{Z}, 1 \leq i \leq N\right\}
$$
and the set of edges
$$
\left\{\{x, y\}: x, y \in \mathbb{Z}^N, \sum_{i=1}^N\left|x_i-y_i\right|=1\right\} .
$$

For $T \in \mathbb{N}$, a function $g$ on $\mathbb{Z}^N$ is called T-periodic if $g\left(x+T e_i\right)=g(x), \forall x \in \mathbb{Z}^N, 1 \leq$ $i \leq N$, where $e_i$ is the unit vector in the $\mathrm{i}$-th coordinate.

The energy functional $J$ associated with problem \eqref{eqlogv} is
$$
J(u):=\frac{1}{2} \int_{\mathbb{V}}\left(|\nabla u|^2+(V(x)+1) u^2\right) d \mu-\frac{1}{2} \int_{\mathbb{V}} u^2 \log u^2 d \mu.
$$
If $G=(\mathbb{V},\mathbb{E})$ is the lattice graph $\mathbb{Z}^N$, then $J(u)$ is not of class $C^{1}$ in $H^1(\mathbb{Z}^N)$. Because there exists a function $u\in H^1(\mathbb{Z}^N)$ such that $\int_{\mathbb{V}} u^2 \log u^2 d \mu=-\infty$, the readers could find an example for $p=2$ in the appendix.

Throughout this paper,  we always assume that $\underset{x\in \mathbb{V}}\inf V(x)> -1$. In fact, $\underset{x\in \mathbb{V}}\inf V(x)>-\infty$ is enough for our study. Let us substitute $u$ with $\lambda v(\lambda>0)$ in \eqref{eqlogv}, one has
\begin{equation}\label{eqlambda}
-\Delta v+\left(V(x)-\log \lambda^2\right) v= v \log v^2, \quad x \in \mathbb{V}.
\end{equation}
Thus $(V(x)+1)>0$ is not essential. By replacing $V(x)$ with $\widehat{V}(x) \equiv V(x)-\log \lambda^2$, we get $\widehat{V}(x)>-1$ for a suitable $\lambda>0$ and thus we can recover the desired result for any $V(x)$ bounded from below.
Naturally, an important problem  is  to investigate the existence of solutions to problem  \eqref{loga} when $\underset{x\in \mathbb{V}}\inf V(x)=-\infty$.

Now it is the position to present our first main result.

\begin{theorem}\label{theo110}
 If $\lim _{d(x,x_0) \rightarrow \infty} V(x) =+\infty$ for some $x_0 \in \mathbb{V}$, then equation \eqref{eqlogv} has infinitely many solutions $\pm u_n\in X$ such that $\lim _{n \rightarrow \infty} J\left( \pm u_n\right)=\infty$.
\end{theorem}

When the potential is bounded, we have the following three results.\\
\begin{theorem}\label{theo120}
Assume that $G=(\mathbb{V},\mathbb{E})$ be the lattice graph $\mathbb{Z}^N$. If the potential $V(x)$  is  $T$-periodic in $x$, then equation \eqref{loga} has a ground state solution $u$. Moreover, $u(x) > 0$ for all $x \in \mathbb{Z}^N$ or $u(x) < 0$ for all $x \in \mathbb{Z}^N$.
\end{theorem}

For the asymptotically periodic potential, we also have a  result similar to Theorem \ref{theo120}.
\begin{theorem}\label{theo130}
Assume that $G=(\mathbb{V},\mathbb{E})$ be the lattice graph $\mathbb{Z}^N$. If the potential $V(x)$ satisfies the following assumptions:
\begin{enumerate}[label=($V_\arabic{*}$), ref=$V_\arabic{*}$]
\item\label{V1}
There exists a $T$-periodic function $V_p$ such that
$$
-1< \inf _{x\in\mathbb{Z}^N} V(x) \leq V(x) \leq V_p(x), \,\,\text{for all}\,\, x \in \mathbb{Z}^N,
$$
and
$$
\left|V(x)-V_p(x)\right| \rightarrow 0 \quad \text { as } \quad|x| \rightarrow+\infty,
$$
\end{enumerate}
then problem \eqref{loga} has a ground state solution $u$. Moreover, $u(x) > 0$ for all $x \in \mathbb{Z}^N$ or $u(x) < 0$ for all $x \in \mathbb{Z}^N$.
\end{theorem}

For the potential well case, in fact, we can regard it as a special case of asymptotic periodic potential.
\begin{theorem}\label{theo140}
 Assume that $G=(\mathbb{V},\mathbb{E})$ be the lattice graph $\mathbb{Z}^N$. If the potential $V$ satisfies the following assumption
 \begin{enumerate}[label=($V_2$), ref=$V_\arabic{*}$]
\item\label{V2}
$$
-1 \leq \inf _{x\in\mathbb{Z}^N} V(x) \leq \sup _{x\in{Z}^N} V(x)=V_{\infty}<\infty \quad \text { with } \quad V_{\infty}=\lim _{|x| \rightarrow \infty} V(x)
$$
\end{enumerate}
then problem \eqref{loga} has a ground state solution $u$. Moreover, $u(x) > 0$ for all $x \in \mathbb{Z}^N$ or $u(x) < 0$ for all $x \in \mathbb{Z}^N$.
\end{theorem}
The  above results can be extended to a quasi-transitive graph  $G$. $G$ is called a quasi-transitive graph if there are finitely many orbits for the action of $Aut(G)$ on $G$ where $Aut(G)$ is the set of automorphisms of $G$.
\begin{theorem}\label{theo150}
\label{theoquasi}
Let $G = (\mathbb{V},\mathbb{E})$ be a quasi-transitive graphs. Suppose $\Gamma \leq Aut(G)$ and the
action of $\Gamma$ on G has finitely many orbits. Then, by replacing $T$-periodicity of $V$ with $\Gamma$-invariance, Theorem \ref{theo120} holds on $G$. By replacing $\abs{x}$ with $d(x,x_0)$ for some $x_0 \in \mathbb{V}$, Theorem \ref{theo130} and \ref{theo140} holds on $G$.
\end{theorem}

The rest of the paper is organized as follows.  In Section 2  we give some
preliminaries on graphs. We build the logarithmic Sobolev inequality on lattice graphs which is useful in the proof of boundedness of  $(PS)$ sequence when the potential is coercive, and we also provide a novel proof for the boundedness of $(PS)$ sequence when the potential is bounded. Moreover, we also summarize some pertinent results and definitions related the logarithmic Schr\"{o}dinger equations, mainly taken from \cite{JS, SS}. In Section 3, we give the proof of Theorem \ref{theo110}.  In Section 4, we prove Theorem \ref{theo120}.  In Section 5, the proofs of Theorem \ref{theo130}, \ref{theo140} and \ref{theo150} are provided. In the final section 6, we briefly sketch how the results of Theorems \ref{theo110}-\ref{theo150} can be generalized to an equation involving the $p$-Laplacian with logarithmic nonlinearity.

\section{The variational framework and some preliminaries}\label{secframe}
 Throughout this paper, let $G$ be a uniformly locally finite graph, i.e.,
$$
\#\{y \in \mathbb{V}: y \sim x\} \leq C, \quad \forall x \in \mathbb{V},
$$
where $C$ is a uniform constant. It is clear that $\mathbb{Z}^N$ is a special case of uniformly locally finite graph with a uniform constant $2N$.

For any function $u, v \in C(\mathbb{V})$, we define the associated gradient form as
$$
\Gamma(u, v)(x):= \frac{1}{2} \sum_{y \sim x} (u(y)-u(x))(v(y)-v(x)) .
$$
Let $\Gamma(u)=\Gamma(u, u)$ and
$$
|\nabla u|(x):=\sqrt{\Gamma(u)(x)}=\left( \frac{1}{2} \sum_{y \sim x}(u(y)-u(x))^2\right)^{1 / 2} .
$$

We shall work in Hilbert space
$$
X=\left\{u \in H^1\left(\mathbb{V}\right) ; \int_{\mathbb{V}} (V(x)+1)u^2 d\mu<\infty\right\}
$$
with the inner product given by
$$
\langle u, v\rangle:=\int_{\mathbb{V}} \Gamma(u, v)+(V(x)+1)u v d \mu .
$$
If the potential $V$ is coercive, $X$ is a proper subspace of $H^{1}\left(\mathbb{V}\right)$ and compactly embedded in $\ell^{p}\left(\mathbb{V}\right)$ for all $2 \leq p \leq \infty$, see the details in \cite[Lemma 2.2]{Gri1}.

If the potential function $V: \mathbb{V} \rightarrow \mathbb{R}$ is bounded uniformly, and  there exist $C_1, C_2>-1$, such that
$$
C_1 \leq V(x) \leq C_2, \quad \forall x \in \mathbb{V},
$$
it is easy to show that the spaces $H^1\left(\mathbb{V}\right)$ and $X$ are equivalent. If $G$ is a uniformly locally finite graph, we may show that the space $H^1(\mathbb{V})$ is equivalent to $\ell^2(\mathbb{V})$.  Therefore, spaces $H^1\left(\mathbb{V}\right)$, $X$ and $\ell^2(\mathbb{V})$ are pairwise equivalent to each other.
\begin{lemma}\label{equi1}
 The spaces $H^1(\mathbb{V})$ and $\ell^2(\mathbb{V})$ are equivalent.
\end{lemma}
\begin{proof}
Let $u \in \ell^2(\mathbb{V})$. For any $x \in \mathbb{V}$, we have
\[
    \sum_{y \sim x} \abs{u(y) - u(x)}^2 \leqslant 2C \abs{u(x)}^2 + 2 \sum_{y \sim x}\abs{u(y)}^2.
\]
Thus,
\[
    \int_{\mathbb{V}}{\abs{\nabla u}^2} \, d\mu = \frac{1}{2}\sum_{x \in \mathbb{V}}\sum_{y \sim x} \abs{u(y) - u(x)}^2 \leqslant C \sum_{x \in \mathbb{V}}\abs{u(x)}^2 + \sum_{x \in \mathbb{V}}\sum_{y \sim x} \abs{u(y)}^2 \leq 2C \sum_{x \in \mathbb{V}} \abs{u(x)}^2.
\]
We conclude that $u \in H^1(\mathbb{V})$ and
\[
     \norm{u}_2 \leq \norm{u} \leq \sqrt{2C+1}\norm{u}_2.
\]
\end{proof}
The following embedding is very important for our problem, or the detail of proof, please refer to \cite[Lemma 2.1]{Gri1}.
\begin{lemma}\label{lemembed}
$X$ is continuously embedded into $\ell^q(\mathbb{V})$ for any $q \in[2, \infty]$. Moreover, for any bounded sequence $\left\{u_k\right\} \subset X$, there exists $u \in X$ such that, up to a subsequence,
$$
\begin{cases}u_k \rightharpoonup u & \text { in } \quad X, \\ u_k(x) \rightarrow u(x) & \forall x \in \mathbb{V}, \\ u_k \rightharpoonup u & \text { in } \quad \ell^q(\mathbb{V}) .\end{cases}
$$
\end{lemma}

To prove Theorem \ref{theo110}, the following logarithmic Sobolev inequality  is crucial for proving the boundedness of $(PS)$ sequences when the potential is coercive.

\begin{lemma}\label{logaine2}
For all $u \in \ell^2(\mathbb{V})$ such that $\norm{u}_2 = 1$, there holds
    \begin{equation}
    \label{eqlogineq2}
        \int_{\mathbb{V}}{u^2\log{u^2}}\, d\mu \leq 0.
    \end{equation}
\end{lemma}
\begin{proof}
Let  $u \in \ell^2(\mathbb{V})$ such that $\norm{u}_2 = 1$. Since $\abs{u(x)}^2 \leq \norm{u}_2^2 = 1$ for all $x \in \mathbb{V}$, we have $u^{2}(x)\log{u^2(x)} \leq 0$ for any $x \in \mathbb{V}$ and $u(x)\neq 0$.  Thus
\[
    \int_{\mathbb{V}}{u^2\log{u^2}}\, d\mu \leq 0.
\]
\end{proof}

It follows from \eqref{eqlogineq2} that, for any $u \in \ell^2(\mathbb{V})$ with $\norm{u}_2 \neq 0$, we have
\begin{equation}
        \label{eqlogineq3}
        \int_{\mathbb{V}} u^2\log{u^2}\, d\mu \leq \norm{u}_2^2\log{\norm{u}_2^2}.
\end{equation}

When the potential $V(x)$ is bounded in $\mathbb{V}$, it is not necessary to need the logarithmic Sobolev inequality in $\mathbb{V}$ to prove the boundedness of $(PS)$ sequences, here we only use the equivalence of spaces $\ell^2(\mathbb{V})$  and $X$. I think that this is a big difference between the the continuous case $\mathbb{R}^N$ and the problem in graphs.

\begin{lemma}\label{bounded}
Assume that the potential $V(x)$ is bounded and satisfies $\inf _{\mathbb{V}} V(x)>-1$. If $\left(u_n\right)$ is a sequence such that $\left(J\left(u_n\right)\right)$ is bounded above and $J^{\prime}\left(u_n\right) \rightarrow 0$, then $\left(u_n\right)$ is bounded. Moreover, since $\Psi \geq 0,\left(J\left(u_n\right)\right)$ is also bounded below and hence it is a $(PS)$ sequence.
\end{lemma}
\begin{proof}
Let $\left(u_n\right) \subset X$ be a sequence with $\left(J\left(u_n\right)\right)$ bounded above and $J^{\prime}\left(u_n\right) \rightarrow 0$ as $n \rightarrow \infty$. Then, choosing $v=u_n$ as test function, we have
\begin{equation}
\label{eqcontroll2}
    \norm{u_n}_2^2 = \int_{\mathbb{V}} {u_n^2} \, d\mu = 2 J\left(u_n\right)-\left\langle J^{\prime}\left(u_n\right), u_n\right\rangle \leq C+o_{n}(1)\norm{u_n} \leqslant  C+o_{n}(1)\norm{u_n}_2, \quad \text { as } n \rightarrow \infty.
\end{equation}
where we use the fact that the spaces $\ell^2(\mathbb{V})$ and $X$ are equivalent, proving the first assertion. Then the second assertion immediately follows.
\end{proof}

Following the approach explored in \cite{JS, SS}, due to the lack of smoothness of $J$, let us decompose it into a sum of a $C^1$ functional and a convex lower semicontinuous functional. For $\delta>0$, let us define the following functions:
$$
F_1(s)= \begin{cases}0, & s=0 \\ -\frac{1}{2} s^2 \log s^2 & 0<|s|<\delta \\ -\frac{1}{2} s^2\left(\log \delta^2+3\right)+2 \delta|s|-\frac{1}{2} \delta^2, & |s| \geq \delta\end{cases}
$$
and
$$
F_2(s)=\frac{1}{2} s^2 \log s^2+F_1(s), \quad s \in \mathbb{R} .
$$

It was proved in \cite{JS, SS} that $F_1$ and $F_2$ verify the following properties:
$$
F_1, F_2 \in C^1(\mathbb{R}, \mathbb{R}) \text {. }
$$
If $\delta>0$ is small enough, $F_1$ is convex, even, $F_1(s) \geq 0$ for all $s \in \mathbb{R}$ and
\begin{equation}\label{use1}
F_1^{\prime}(s) s \geq 0, s \in \mathbb{R} .
\end{equation}

Fix $p \in (2, +\infty)$, there exists $C>0$ such that
\begin{equation}\label{use2}
\left|F_2^{\prime}(s)\right| \leq C|s|^{p-1}, \quad \forall s \in \mathbb{R} .
\end{equation}

Now, we define
$$
\Phi(u)=\frac{1}{2} \int_{\mathbb{V}}\left(|\nabla u|^2+(V(x)+1)\right)|u|^2) d \mu-\int_{\mathbb{V}} F_2(u) d \mu,
$$
and
$$
\Psi(u)=\int_{\mathbb{V}} F_1(u) d \mu,
$$
then
\begin{equation}\label{add1}
J(u)=\Phi(u)+\Psi(u), \quad u \in X.
\end{equation}

Using the above information, it follows that $\Phi \in C^1\left(X, \mathbb{R}\right), \Psi$ is convex and lower semicontinuous, but $\Psi$ is not a $C^1$ functional, in fact, it is not continuous in $X$. Due to this fact, we will look for a critical point in the sub-differential. Here we state some definitions that can be found in \cite{S}.

\begin{definition}
Let $E$ be a Banach space, $E^{\prime}$ be the dual space of $E$ and $\langle\cdot, \cdot\rangle$ be the duality paring between $E^{\prime}$ and $E$. Let $J: E \rightarrow \mathbb{R}$ be a functional of the form $J(u)=\Phi(u)+\Psi(u)$, where $\Phi \in C^1(E, \mathbb{R})$ and $\Psi$ is convex and lower semicontinuous. Let us list some definitions:\\
(i) The sub-differential $\partial J(u)$ of the functional $J$ at a point $u \in E$ is the following set
$$
\left\{w \in E^{\prime}:\left\langle\Phi^{\prime}(u), v-u\right\rangle+\Psi(v)-\Psi(u) \geq\langle w, v-u\rangle, \forall v \in E\right\} .
$$
(ii) A critical point of $J$ is a point $u \in E$ such that $J(u)<+\infty$ and $0 \in \partial J(u)$, i.e.
$$
\left\langle\Phi^{\prime}(u), v-u\right\rangle+\Psi(v)-\Psi(u) \geq 0, \forall v \in E .
$$
(iii) A $(PS)$ sequence at level $d$ for $J$ is a sequence $\left(u_n\right) \subset E$ such that $J\left(u_n\right) \rightarrow d$ and there is a numerical sequence $\tau_n \rightarrow 0^{+}$with
\begin{equation}\label{eqps}
    \left\langle\Phi^{\prime}\left(u_n\right), v-u_n\right\rangle+\Psi(v)-\Psi\left(u_n\right) \geq-\tau_n\left\|v-u_n\right\|, \forall v \in E .
\end{equation}
(iv) The functional $J$ satisfies the $(PS)$ condition if all $(PS)$ sequences has a convergent subsequence.\\
(v) The effective domain of $J$ is the set $D(J)=\{u \in E: J(u)<+\infty\}$.
\end{definition}

To proceed further we gather and state below some useful results that leads to a better understanding of the problem and of its particularities. In what follows, for each $u \in D\left(J\right)$, we set the functional $J^{\prime}(u): C_c(\mathbb{V}) \rightarrow \mathbb{R}$ given by
$$
\left\langle J^{\prime}(u), z\right\rangle=\left\langle\Phi^{\prime}(u), z\right\rangle-\int_{\mathbb{V}} F_1^{\prime}(u) z d \mu, \quad \forall z \in C_c(\mathbb{V})
$$
and define
$$
\left\|J^{\prime}(u)\right\|=\sup \left\{\left\langle J^{\prime}(u), z\right\rangle: z \in C_c(\mathbb{V}) \text { and }\|z\| \leq 1\right\} .
$$

If $\left\|J^{\prime}(u)\right\|$ is finite, then $J^{\prime}(u)$ may be extended to a bounded operator in $X$, and so, it can be seen as an element of $X^{\prime}$.

\begin{lemma}\label{sub1}
 Let $J$ satisfy \eqref{add1}, then:\\
(i) If $u \in D\left(J\right)$ is a critical point of $J$, then
$$
\left\langle\Phi^{\prime}(u), v-u\right\rangle+\Psi(v)-\Psi(u) \geq 0, \quad \forall v \in X,
$$
or equivalently
$$
\begin{aligned}
& \int_{\mathbb{V}}\nabla u \nabla(v-u) d \mu+\int_{\mathbb{V}}(V(x)+1) u(v-u) d \mu+\int_{\mathbb{V}} F_1(v) d x-\int_{\mathbb{V}} F_1(u) d \mu \\
& \geq \int_{\mathbb{V}} F_2^{\prime}(u)(v-u) d \mu, \,\,\forall v \in X.
\end{aligned}
$$
(ii) For each $u \in D\left(J\right)$ such that $\left\|J^{\prime}(u)\right\|<+\infty$, we have $\partial J(u) \neq \emptyset$, that is, there is $w \in X^{\prime}$, which is denoted by $w=J^{\prime}(u)$, such that
$$
\left\langle\Phi^{\prime}(u), v-u\right\rangle+\int_{\mathbb{V}} F_1(v) d \mu-\int_{\mathbb{V}}F_1(u) d\mu \geq\langle w, v-u\rangle, \quad \forall v \in X, \quad(\operatorname{see} \text{\cite{SS}})
$$
(iii) If a function $u \in D\left(J\right)$ is a critical point of $J$, then $u$ is a solution of (1.3) [(i) in Lemma 2.4, \text{\cite{JS}}].\\
(iv) If $\left(u_n\right) \subset X$ is a $(PS)$ sequence, then
$$
\left\langle J^{\prime}\left(u_n\right), z\right\rangle=o_n(1)\|z\|, \quad \forall z \in C_c(\mathbb{V}).
$$
[see (ii) in Lemma 2.4, \text{\cite{JS}}].\\
(v) If $\Omega$ is a bounded domain, then $\Psi$ (and hence $J$) is of class $C^1$ in $H^1(\Omega)$
(Lemma 2.2 in\text{\cite{SS}}). More precisely, the functional
$$
\Psi(u)=\int_{\Omega} F_1(u) d \mu, \quad \forall u \in H^1(\Omega)
$$
belongs to $C^1\left(H^1(\Omega), \mathbb{R}\right)$.\\
(vi)  If $J(u_n)$ is bounded above, $J'(u_n) \to 0$ and $u_n \rightharpoonup u$, then $u$ is a critical point of $J$.
\end{lemma}
\begin{remark}\label{repoint}
Let $u$ be a weak solution of \eqref{loga1}. Due to the formula for integral by parts on graphs, see e.g. \cite[Theorem 2.1]{Gr}, for any test function $\varphi \in C_c(\mathbb{V})$, we have $\int_{\mathbb{V}}(-(\Delta u)\varphi +V(x)u \varphi) d \mu=\int_{\mathbb{V}}u\varphi\log{u^2} d \mu$. For any fixed $y \in \mathbb{V}$, take the test function to be
$$
\delta_y(x)= \begin{cases} 1 , & x=y, \\ 0, & x \neq y .\end{cases}
$$
Then we get $-\Delta u\left(y\right)+V(y)u\left(y\right)=u\left(y\right)\log{u^2(y)}$, which tells us that, in fact, a weak solution is also a point-wise solution.
\end{remark}
\begin{definition}\label{defcompact}
    We shall say that a set $A \subset X$ has compact support if there exists $R>0$ and $x_0 \in \mathbb{V}$ such that $u(x)=0$ for all $d(x,x_0)>R$ and $u \in A$.
\end{definition}

\section{The compact case}

\begin{proposition}
\label{propcompactps}
If $\left(u_{n}\right)$ is a sequence such that $J\left(u_{n}\right)$ is bounded above and $J^{\prime}\left(u_{n}\right) \rightarrow 0$, then $\left(u_{n}\right)$ has a convergent subsequence. In particular, $J$ satisfies the $(PS)$ condition.
\end{proposition}
\begin{proof}
    First we show that $\left(u_{n}\right)$ is bounded in $X$. Choose $d \in \mathbb{R}$ such that $J\left(u_{n}\right) \leq d$ for all $n$. Then

\[
\norm{u_{n}}_{2}^{2}=\int_{\mathbb{V}} u_{n}^{2} d\mu=2 J\left(u_{n}\right)-\left\langle J^{\prime}\left(u_{n}\right), u_{n}\right\rangle \leq 2 d+o(1)\left\|u_{n}\right\| \quad \text { as } n \rightarrow \infty.
\]
By \eqref{eqlogineq3}, we obtain
$$
\begin{aligned}
2 d \geq 2 J\left(u_{n}\right) &=\left\|u_{n}\right\|^{2}-\int_{\mathbb{V}} u_{n}^{2} \log u_{n}^{2} d\mu \\
& \geq \left\|u_{n}\right\|^{2}- \left\|u_{n}\right\|_{2}^{2}\log {\left\|u_{n}\right\|_{2}^{2}} \geq \left\|u_{n}\right\|^{2}-C_{1}\left(1+\left\|u_{n}\right\|^{r}\right),
\end{aligned}
$$
where we take $r \in(1,2)$. Hence the sequence $\left(u_{n}\right)$ is bounded in $X$.

Passing to a subsequence, $u_{n} \rightharpoonup u$ in $X$ for some $u\in X$ and since the embedding $X \hookrightarrow \ell^{p}\left(\mathbb{V}\right)$ is compact for $p \in\left[ 2,\infty \right]$ from \cite[Lemma 2.2]{Gri1}, $u_{n} \rightarrow u$ in $\ell^{p}\left(\mathbb{V}\right)$ for all such $p$. Taking $v=u$ in \eqref{eqps} gives
\begin{equation}\label{eqps3}
    \left\langle u_{n}, u-u_{n}\right\rangle - \int_{\mathbb{V}} F_{2}^{\prime}\left(u_{n}\right)\left(u-u_{n}\right) d\mu+\Psi(u)-\Psi\left(u_{n}\right) \geq-\tau_{n}\left\|u-u_{n}\right\|,
\end{equation}
thus
\begin{equation}\label{eqpscon}
    \|u\|^{2}-\norm{u_{n}}^{2}+\Psi(u)-\Psi\left(u_{n}\right)+o(1) \geq o(1)
\end{equation}
Since $\lim \inf _{n \rightarrow \infty} \Psi\left(u_{n}\right) \geq \Psi(u)$ and $\liminf _{n \rightarrow \infty}\left\|u_{n}\right\|^{2} \geq\|u\|^{2}$, the inequality above implies $\left\|u_{n}\right\| \rightarrow\|u\|$ and hence $u_{n} \rightarrow u$ in $X$.
\end{proof}

Similar to  \cite[Section 3]{JS}, we shall prove Theorem \ref{theo110} by adapting the arguments of Bartsch's Fountain Theorem \cite[Theorem 2.5]{Ba}
and \cite[Theorem 3.6]{W}. The following results are very useful for the proof of Theorem \ref{theo110}, the detailed proofs of which can be found in \cite{JS}, here we omit them.

\begin{lemma}\label{lemempty1}
If $K_{d}=\emptyset$, then there exists $\varepsilon_{0}>0$ such that there are no $(PS)$ sequences in $J_{d-2 \varepsilon_{0}}^{d+2 \varepsilon_{0}}$.
\end{lemma}

\begin{proposition}\label{propdefor1}
    Suppose $K_{d}=\emptyset$ and $\varepsilon_{0}$ is as in Lemma \ref{lemempty1}. If $\varepsilon \in\left(0, \varepsilon_{0}\right)$, then for each compact set $A \subset J^{d+\varepsilon} \cap C_c\left(\mathbb{V}\right)$,  there exists $T>0$ such that $J(\eta(T, A)) \subset J^{d-\varepsilon}$ where the flow $\eta$ given by (3.5) in \cite{JS}.
\end{proposition}

Since $X$ is separable and $C_c\left(\mathbb{V}\right)$ is dense in $X$, there exists an increasing sequence $\left(X_{k}\right) \subset C_c\left(\mathbb{V}\right)$ of subspaces such that $\operatorname{dim} X_{k}=k$ and $X=\overline{\cup_{k=1}^{\infty} X_{k}}$. Let $Z_{k}:=X_{k}^{\perp}$ and
$$
B_{k}:=\left\{u \in X_{k}:\|u\| \leq \rho_{k}\right\}, \quad N_{k}:=\left\{u \in Z_{k-1}:\|u\|=r_{k}\right\}, \quad \text { where } \rho_{k}>r_{k}>0.
$$
\begin{lemma}\label{lemgamma}(See \cite[Lemma 3.4]{W}.) If $\gamma \in C\left(B_{k}, X\right)$ is odd and $\left.\gamma\right|_{\partial B_{k}}=i d$, then $\gamma\left(B_{k}\right) \cap N_{k} \neq \emptyset$.
\end{lemma}
\begin{proposition}\label{propak}
    There exist $\rho_{k}>r_{k}>0$ such that
$$
a_{k}:=\max _{\substack{u \in X_{k} \\\|u\|=\rho_{k}}} J(u) \leq 0 \text { for all } k \quad \text { and } \quad b_{k}:=\inf _{u \in N_{k}} J(u) \rightarrow \infty \text { as } k \rightarrow \infty.
$$
\end{proposition}
\begin{proof}
Let $u=s w$, where $w \in X_{k}$ and $\|w\|_{2}=1$. Then
$$
J(s w)=\frac{s^{2}}{2}\left(\int_{\mathbb{V}}\left(|\nabla w|^{2}+(V(x)+1) w^{2}\right) d\mu-\log s^{2}-\int_{\mathbb{V}} w^{2} \log w^{2} d\mu\right) .
$$
Since all norms in $X_{k}$ are equivalent and $X_{k} \subset C_c\left(\mathbb{V}\right)$, both integrals above are uniformly bounded. Hence $J(s w) \rightarrow-\infty$ uniformly in $w$ as $s \rightarrow \infty$, so there exists $\rho_{k}$ such that $a_{k} \leq 0$. Moreover, $\rho_{k}$ may be chosen as large as we need.

Let
$$
\beta_{k}:=\sup _{\substack{u \in Z_{k-1} \\\|u\|=1}}\|u\|_{2}.
$$
Then $\beta_{k} \rightarrow 0$. Because the sequence $\left(\beta_{k}\right)$ is positive and decreasing, hence $\beta_{k} \rightarrow \beta \geq 0$ and $\left\|u_{k}\right\|_{2} \geq \beta_{k} / 2$ for some $u_{k} \in Z_{k-1}$, $\left\|u_{k}\right\|=1$.
Since $u_{k} \rightharpoonup 0$ in $X, u_{k} \rightarrow 0$ in $\ell^{2}\left(\mathbb{V}\right)$. This implies that $\beta=0$.

Using \eqref{eqlogineq3} as in the proof of Proposition \ref{propcompactps} we obtain
$$
\begin{aligned}
J(u) & =\frac{1}{2}\|u\|^{2}-\frac{1}{2} \int_{\mathbb{R}^{N}} u^{2} \log u^{2} d\mu\\
& \geq \frac{1}{2}\|u\|^{2}-\|u\|_{2}^{2}\log \|u\|_{2}^{2} \geq \frac{1}{2}\|u\|^{2}-C_{1}\|u\|_{2}^{p}-C_{2} \\
& \geq \frac{1}{2}\|u\|^{2}-C_{1} \beta_{k}^{p}\|u\|^{p}-C_{2},
\end{aligned}
$$
where $p > 2$. Let $r_{k}=1 / \beta_{k}$ and $\|u\|=r_{k}$. Then
$$
J(u) \geq \frac{1}{2 \beta_{k}^{2}}-C_{1}-C_{2} \rightarrow \infty \quad \text { as } k \rightarrow \infty
$$
and hence $b_{k} \rightarrow \infty$. Since we may choose $\rho_{k}>r_{k}$, the proof is complete.
\end{proof}
\begin{proof}[Proof of Theorem \ref{theo110}]
Let
$$
\Gamma_{k}:=\left\{\gamma \in C\left(B_{k}, X\right): \gamma \text { is odd, }\left.\gamma\right|_{\partial B_{k}}=i d \text { and } \gamma\left(B_{k}\right) \text { has compact support }\right\}
$$
and
$$
d_{k}:=\inf _{\gamma \in \Gamma_{k}} \max _{u \in B_{k}} J(\gamma(u))
$$
From Lemma \ref{lemgamma}, $\gamma\left(B_{k}\right) \cap N_{k} \neq \emptyset$. Thus,  $d_{k} \geq b_{k} \rightarrow \infty$ and it remains to show that $K_{d_{k}} \neq \emptyset$ if $k$ is large. Argue by contradiction, choose $\varepsilon_{0}, \varepsilon$ and $T$ as in Proposition \ref{propdefor1} and let $\gamma \in \Gamma_{k}$ be such that $\gamma\left(B_{k}\right) \subset J^{d_{k}+\varepsilon}$. Let $\beta(u):=\eta(T, \gamma(u))$, where the flow $\eta$ given by (3.5) in \cite{JS}. Since $\eta(T, u)=u$ for all $u \in J^{d_{k}-\varepsilon_{0}}$, $\beta \in \Gamma_{k}$. By Proposition \ref{propdefor1}, $\beta\left(B_{k}\right) \subset J^{d_{k}-\varepsilon}$, which contradicts with the definition of $d_{k}$. The proof of Theorem \ref{theo110} is complete.
\end{proof}

\section{The periodic case}

First of all, we show that the functional $J$ satisfies the mountain pass geometry \cite{W}.

\begin{lemma}\label{mp}
The functional $J$ satisfies the following conditions:\\
	(i) there exist $\alpha, \rho>0$ such that $J(u) \geq \alpha$ for any $u \in X$ with $\|u\|=\rho$;\\
	(ii)  there exists $e \in X$ with $\|e\|>\rho$ such that $J(e)<0$.
\end{lemma}
\begin{proof}
(i) Since $F_1(s) \geq 0$ for all $s \in \mathbb{R}$, we have
$$
J(u) \geq \frac{1}{2}\|u\|^2-\int_{\mathbb{V}} F_2(u) d \mu.
$$
Now, from \eqref{use2} and Lemma \ref{lemembed}, fixed $p \in(2, +\infty)$, it follows that
$$
J(u) \geq \frac{1}{2}\|u\|^2-C\|u\|^p \geq \alpha>0,
$$
for some $\alpha>0$ and $\|u\|=\rho>0$ small enough.

(ii)Fix $u\in D(J) \backslash\{0\}$. Then we have
$$
J(s u)=J(u) s^2-s^2 \log s \int_{\mathbb{V}}u^2 d\mu = s^2\left(J(u)-\log{s}\int_{\mathbb{V}}u^2 d\mu\right)
$$
for all $s >0$. Thus, $J(su) \to -\infty$ as $s \to +\infty$.
\end{proof}

From Lemma \ref{mp}, we can define
$$
\Gamma:=\{\gamma\in C([0,1], X): \gamma(0)=0, J(\gamma(1))<0\}
$$
and
$$
c:=\inf _{\gamma\in \Gamma} \sup _{s \in[0,1]} J(\gamma(s)).
$$

Now, let
$$
\mathcal{N}:=\left\{u \in D(J) \backslash\{0\}:\left\langle J^{\prime}(u), u\right\rangle=0\right\}
$$
be the Nehari manifold for $J$ and
$$
 c_{\mathcal{N}}:=\inf _{u \in \mathcal{N}} J(u).
$$
Now we show the relation between the minimax level of $J$ and ground state energy of $J$ which is useful below.
\begin{lemma}\label{equa}
$
c= c_{\mathcal{N}}.
$
\end{lemma}

\begin{proof}

Let $u \in D(J) \backslash\{0\}$. Then the map $s \mapsto J(s u)$ admits a unique maximum point on $(0, \infty)$. In fact, if $\varphi:(0, \infty) \rightarrow \mathbb{R}$ is the map defined by
$$
\varphi(s):=J(s u)=J(u) s^2-s^2 \log s \int_{\mathbb{V}}u^2 d \mu, \quad s>0,
$$
we have $\varphi(s)>0$ for $s>0$ sufficiently small and $\varphi(s)<0$ for all $s>0$ large enough. Moreover, $\varphi^{\prime}(s)=0$ with $s>0$ if and only if
\begin{equation}\label{eqnehari}
    J(u)=\frac{2 \log s+1}{2} \int_{\mathbb{V}} u^2 d \mu,
\end{equation}
which proves the claim.

Let $u \in \mathcal{N}$ and let us consider $J\left(t^* u\right)<0$, for some $t^*>0$. If $\gamma :[0,1] \rightarrow X$ is the continuous path $\gamma(t)=t \cdot t^* u$, then
$$
c \leq \max _{t \in[0,1]} J\left(\gamma (t)\right) \leq \max _{t \geq 0} J(t u)=J\left(t u\right)=J(u)
$$
and consequently $c \leq c_{\mathcal{N}}$.

To prove the reverse inequality, by \cite[Theorem 3.1]{AD}, there exists a $(PS)$ sequence $\left(u_n\right) \subset X$ for $J$. By Lemma \ref{bounded} , the sequence $\left(u_n\right)$ is bounded in $X$. A key point in our approach is the following result
$$
u_n\nrightarrow 0,\,\,\text{in}\,\, \ell^2(\mathbb{V}).
$$
Indeed, on the contrary, by interpolation, we would have  $u_n \rightarrow 0$ in $\ell^p(\mathbb{V})$, for some $p \in(2, +\infty)$, then, by \eqref{use2} we have
\begin{equation}\label{zero1}
\int_{\mathbb{V}}  F_2^{\prime}\left(u_n\right) u_n d \mu \rightarrow 0 .
\end{equation}
On the other hand, by \eqref{zero1}, we have
$$
\left\|u_n\right\|^2+\int_{\mathbb{V}} F_1^{\prime}\left(u_n\right) u_n d \mu=\left\langle J^{\prime}\left(u_n\right) u_n\right\rangle + \int_{\mathbb{V}} F_2^{\prime}\left(u_n\right) u_n d \mu=o_n(1)\left\|u_n\right\|+\int_{\mathbb{V}} F_2^{\prime}\left(u_n\right) u_n d \mu=o_n(1),
$$
from where it follows that $u_n \rightarrow 0$ in $X$ and $F_1^{\prime}\left(u_n\right) u_n \rightarrow 0$ in $\ell^1(\mathbb{V})$. Since $F_1$ is convex, even and $F(t) \geq F_1(0)=0$ for all $t \in \mathbb{R}$, we derive that $0 \leq F_1(t) \leq F_1^{\prime}(t) t$ for all $t \in \mathbb{R}$. Hence, $F_1\left(u_n\right) \rightarrow 0$ in $\ell^1(\mathbb{V})$, and so, $J\left(u_n\right) \rightarrow J(0)=0$, but this contradicts the fact that $c>0$. Hence there are constants $a$ and $b$ such that
\begin{equation}\label{abbe}
0<a \leq\norm{u_n}_2 \leq b, \quad \forall n \in \mathbb{N} .
\end{equation}

For each $u_n$, let $t_n>0$ be such that $t_n u_n \in \mathcal{N}$. Recalling that
\begin{equation}\label{equa001}
J\left(t_n u_n\right)=\frac{1}{2} \int_{\mathbb{V}}\left|t_n u_n\right|^2 d \mu, \quad \forall n \in \mathbb{N}
\end{equation}
or equivalently
$$
\int_{\mathbb{V}}\Big(\left|\nabla u_n\right|^2+\left(V(x)+1\right)\left|u_n\right|^2\Big)\ d \mu-\int_{\mathbb{V}} u_n^2 \log \left|t_n u_n\right|^2d \mu=\int_{\mathbb{V}}\left|u_n\right|^2 d \mu, \quad \forall n \in \mathbb{N}
$$
and
$$
\left\langle J^{\prime}\left(u_n\right), u_n \right\rangle = \int_{\mathbb{V}}\Big(\left|\nabla u_n\right|^2+V( x)\left|u_n\right|^2\Big)d \mu-\int_{\mathbb{V}} u_n^2 \log u_n^2 d \mu=o_n(1),
$$
we have that
$$
o_n(1)=2\norm{u_n}_2^2 \log t_n
$$
This equality combines with \eqref{abbe} to result $t_n \rightarrow 1$. On the other hand, by \eqref{equa001} and \eqref{eqcontroll2},
$$
\inf _{u \in \mathcal{N}} J(u) \leq J\left(t_n u_n\right)=\frac{t_n^2}{2} \int_{\mathbb{V}}\left|u_n\right|^2 d \mu=t_n^2\left(J\left(u_n\right)+o_n(1)\left\|u_n\right\|\right)=t_n^2 J\left(u_n\right)+o_n(1) .
$$
Passing to the limit in this inequality, the reverse inequality $c_{\mathcal{N}} \leq c$ holds.
\end{proof}

Now, we prove Theorem \ref{theo120} on lattice graph $\mathbb{Z}^N$.
\begin{proof}[Proof of Theorem \ref{theo120}]

From Lemma \ref{mp} and Lemma \ref{equa}, there exists a $(PS)$ sequence $\left(u_n\right) \subset X$ for $J$ with the level $c$. Moreover, from Lemma \ref{bounded} , $\left(u_n\right)$ is bounded in $X$.

If $u_n \rightarrow 0$ in $\ell^p(\mathbb{Z}^N)$, for some $p \in(2, +\infty)$, similar to the proof of Lemma \ref{equa}, we have
$u_n \rightarrow 0$ in $X$ and $F_1\left(u_n\right) \rightarrow 0$ in $\ell^1(\mathbb{Z}^N)$. Thus, $J\left(u_n\right) \rightarrow 0$, which contradicts  $c>0$.

Thus,  $u_n \nrightarrow 0$ in $\ell^p(\mathbb{Z}^N)$, for $p>2$.  From the boundedness of the sequence $\{u_n\}$ in $\ell^p(\mathbb{Z}^N)$, we have
\begin{equation}\label{03}
\varlimsup_n\left\|u_n\right\|_p=c_{1}>0
\end{equation}
for some positive constant $c_{1}$. By the interpolation inequality, we have
$$
\left\|u_n\right\|_p \leq\left\|u_n\right\|_2^{\frac{2}{p}}\left\|u_n\right\|_{\infty}^{\frac{p-2}{p}}\leq c_2^{\frac{2}{p}}\left\|u_n\right\|_{\infty}^{\frac{p-2}{p}},
$$
where $c_{2}>0$ is a constant. Thus, there exists $\eta>0$ such that
$$
\lim _{n \rightarrow \infty}\left\|u_n\right\|_{\infty}\geq \eta.
$$
Thus, there exists a sequence of $\left(y_n\right) \in \mathbb{Z}^N$ such that
$$
\vert u_n(y_n)\vert\geq \frac{\eta}{2}
$$
 for all $n\in \mathbb{N}$. Now, we define $\tilde{u}_n(x)=u_n\left(x+k_n T\right)$ with $k_n=\left(k_n^1, \ldots, k_n^N\right)$ to ensure that $\left\{y_n-k_n T\right\} \subset \Omega$ where $\Omega=[0, T)^N \cap \mathbb{Z}^N$ is a bounded domain in $\mathbb{Z}^N$. Since $V(x)$ is  $T$-periodic in $x$, $J$ and $J^{\prime}$  are invariant under the translation.  Thus, we have
$$
J\left(\tilde{u}_n\right) \rightarrow c,\,\, J^{\prime}\left(\tilde{u}_n\right) \rightarrow 0 \,\, \text { and } \,\, \tilde{u}_n \rightarrow \tilde{u} \neq 0 \,\, \text{pointwise in}\,\,  \mathbb{Z}^N.
$$
Moreover, from Lemma \ref{sub1}(vi),  we know that $\tilde{u}$ is a nontrivial critical point of $J$. Therefore,
$$
\begin{aligned}
2 c & \leq 2 J(\tilde{u} )=2 J(\tilde{u})-\left\langle J^{\prime}(\tilde{u}),\tilde{u}\right\rangle=\int_{\mathbb{Z}^N}|\tilde{u}|^2 d\mu\leq \int_{\mathbb{Z}^N} |\tilde{u}|^2 d\mu \\
& \leq \liminf _{n \rightarrow+\infty} \int_{\mathbb{Z}^N}\left|\tilde{u}_n\right|^2 d\mu \leq \limsup _{n \rightarrow+\infty} \int_{\mathbb{Z}^N}\left|\tilde{u}_n\right|^2 d\mu=2 c
\end{aligned}
$$
from where it follows that
$$
\tilde{u}_n \rightarrow \tilde{u} \quad \text { in } \quad \ell^2\left(\mathbb{Z}^N\right) .
$$
Now, by interpolation inequality on the Lebesgue spaces, we get
$$
\tilde{u}_n \rightarrow \tilde{u} \,\,\text { in } \,\, \ell^p\left(\mathbb{Z}^N\right), \,\,\forall p \in [2, +\infty) .
$$
By \eqref{use2}, one has
\begin{equation}\label{0077}
\int_{\mathbb{Z}^N} F_2^{\prime}\left(\tilde{u}_n\right) \tilde{u}_n d\mu \rightarrow \int_{\mathbb{Z}^N} F_2^{\prime}\left(\tilde{u}\right) \tilde{u}d\mu\,\,\, \text {as }\,\, n \rightarrow \infty .
\end{equation}
Moreover, using the fact $\left\langle J^{\prime}\left(\tilde{u}_n\right), \tilde{u}_n\right\rangle=o_n(1)$, that is,
$$
\int_{\mathbb{Z}^N}\left(\left|\nabla \tilde{u}_n\right|^2+(V(x)+1)\left|\tilde{u}_n\right|^2\right) d\mu+\int_{\mathbb{Z}^N} F_1^{\prime}\left(\tilde{u}_n\right) \tilde{u}_n d\mu=\int_{\mathbb{Z}^N} F_2^{\prime}\left(\tilde{u}_n\right) \tilde{u}_nd\mu+o_n(1) .
$$
From $\left\langle J^{\prime}\left(\tilde{u}\right), \tilde{u}\right\rangle=0$ and \eqref{0077}, we deduce that
$$
\begin{aligned}
&\int_{\mathbb{Z}^N}\left(\left|\nabla \tilde{u}_n\right|^2+(V(x)+1)\left|\tilde{u}_n\right|^2\right) d\mu+\int_{\mathbb{Z}^N} F_1^{\prime}\left(\tilde{u}_n\right) \tilde{u}_n d\mu\\
= & \int_{\mathbb{Z}^N}\left(\left|\nabla \tilde{u}\right|^2+(V(x)+1)\left|\tilde{u}\right|^2\right) d\mu+\int_{\mathbb{Z}^N} F_1^{\prime}\left(\tilde{u}\right) \tilde{u}d\mu+o_n(1),
\end{aligned}
$$
from where it follows that, for some subsequence,
$$
\tilde{u}_n \rightarrow \tilde{u}\,\, \text { in } \,\, X
$$
and
$$
F_1^{\prime}\left(\tilde{u}_n\right) \tilde{u}_n \rightarrow F_1^{\prime}\left(\tilde{u}\right) \tilde{u} \,\, \text { in } \,\, \ell^1\left(\mathbb{Z}^N\right).
$$
Since $F_1$ is convex, even and $F(0)=0$, we have that $F_1^{\prime}(t) t \geq F_1(t) \geq 0$ for all $t \in \mathbb{R}$. Thus, the last limit together with Lebesgue dominated convergence theorem yields
$$
F_1\left(\tilde{u}_n\right) \rightarrow F_1\left(\tilde{u}\right) \,\,\text {in } \,\, \ell^1\left(\mathbb{Z}^N\right).
$$
Thus, $\tilde{u}$ is a critical point of $J$ with $J(\tilde{u})=c$.

Finally, we show that the ground state solution $u$  has the constant sign, that is, $u(x)>0$ for all $x\in \mathbb{Z}^N$ or $u(x)<0$ for all $x\in \mathbb{Z}^N$. First of all,  since $\abs{u(x) - u(y)}^2 \geq u^2(x) + u^2(y)$ for all $xy \in \mathbb{E}$ with $u(x)u(y) \leq 0$, we have
    \[
            J(u) \geq J(u^{+}) + J(u^{-})
    \]
    for all $u \in D(J)$, where $u^{+}= \max\{u, 0\}$ and $u^{-} = \min\{u, 0\}$.

   Let $u$ be the ground state solution of \eqref{loga}. Argue by contradiction, suppose $u^{+}, u^{-}\neq 0$. From \eqref{equa001}, there exists unique  $t_\pm$ be such that $t_\pm u^{\pm}\in \mathcal{N}$. Recalling that $J(u) = \frac{1}{2}\int_{\mathbb{Z}^N} u^2 d \mu$ and $J(t_\pm u^{\pm}) = \frac{t_\pm^2}{2}\int_{\mathbb{Z}^N} \abs{u^{\pm}}^2 d \mu$. Then we obtain $t_\pm>1$, since $J(t_\pm u^{\pm}) \geq c_\mathcal{N} = J(u)$ and $\int_{\mathbb{Z}^N} u^2 d \mu >\int_{\mathbb{Z}^N} \abs{u^{\pm}}^2 d \mu$. Moreover, by \eqref{eqnehari}, we have $J(u^{\pm})=\frac{2 \log t_\pm+1}{2} \int_{\mathbb{Z}^N} \abs{u^{\pm}}^2 d \mu$. Thus
    \[
            J(u) < J(u^{+}) + J(u^{-}),
    \]
     which is a contradiction. Thus, $u$ does not change sign.

     If $u(x_0)=0$ for some fixed $x_0 \in \mathbb{Z}^N$,
    then \eqref{loga} yields that
    $u(x) = 0$ for all $x \sim x_0$. Since $\mathbb{Z}^N$ is connected, we obtain $u(x) \equiv 0$, which leads to a contradiction with $J(u) = c_\mathcal{N} > 0$.
\end{proof}

\begin{remark}\label{positi}
In the continuous case $\mathbb{R}^N$, from \cite{SS},  we know that the ground state solution $u$ for problem \eqref{loga1} has constant sign.
Assume that $u^{+}=\max\{u, 0\}$ and $u^{-}=\max\{-u, 0\}$, then $u=u^{+}-u^{-}$. By direct computation,  $J(u)=J\left(u^{+}\right)+J\left(u^{-}\right)$ and $u^{+}, u^{-} \in D(J)$. Moreover, let $\mathcal{N}$ be Nehari manifold and $c$ be ground state energy of problem \eqref{loga1}, $0=\left\langle J^{\prime}(u), u^{+}\right\rangle=\left\langle J^{\prime}\left(u^{+}\right), u^{+}\right\rangle$, so either $u^{+} \in \mathcal{N}$ or $u^{+}=0$. A similar conclusion holds for $u^{-}$. Hence, either one of the functions $u^{+}, u^{-}$is equal to 0 or $J(u) \geq 2 c$, which yields a contradiction. Suppose $u=u^{+}$. Then, by a slight variant of the argument in \cite[Section 3.1]{DMS} it follows from the maximum principle (see \cite[Theorem 1]{V}) that $u(x)>0$, for a.e. $x \in \mathbb{R}^N$. However, for problem \eqref{loga} on lattice graph, the previous discussion breaks down. This is mainly because, in general,  $ \int_{\mathbb{Z}^N}|\nabla u|^2 d \mu=\int_{\mathbb{Z}^N}|\nabla (u^{+}-u^{-})|^2 d \mu\neq \int_{\mathbb{Z}^N}|\nabla (u^{+})|^2 d \mu+\int_{\mathbb{Z}^N}|\nabla (u^{-})|^2 d \mu$, further leading to $J(u)\neq J\left(u^{+}\right)+J\left(u^{-}\right)$. Therefore, to show that the ground state solution on the graph has the constant sign, our proof completely differs from the continuous case.
\end{remark}

\section{The asymptotically periodic case and the potential well case}
In this section, we firstly consider  the existence of ground states for problem \eqref{loga} with the asymptotically periodic potential $V$.
Let
$$
\Gamma:=\{\gamma \in C([0,1], X): \gamma(0)=0, J(\gamma(1))<0\},
$$
and
$$
c:=\inf _{\gamma \in \Gamma} \sup _{s \in[0,1]} J(\gamma(s)), \quad c_{\mathcal{N}}:=\inf _{u \in \mathcal{N}} J(u).
$$
Clearly, $c \leq c_{\mathcal{N}}$ by $(V_{1})$.  Replacing $V$ by $T$-periodic potential  $V_p$, we obtain a limiting problem
\begin{equation}\label{equp}
    -\Delta u+V_{p}(x) u=u \log u^{2}, \quad x \in \mathbb{Z}^{N}.
\end{equation}
The energy functional corresponding to \eqref{equp} is $J_p: H^1\left(\mathbb{Z}^N\right) \rightarrow(-\infty,+\infty]$ defined by
\[
J_p(u)=\frac{1}{2} \int_{\mathbb{Z}^N}\left(|\nabla u|^2+(V_p(x)+1) u^2\right) d \mu - \frac{1}{2} \int_{\mathbb{Z}^N} u^2 \log u^2 d \mu,
\]
and the associated Nehari manifold is
\[
\mathcal{N}_p := \left\{u \in D(J_p) \backslash\{0\}:\left\langle J_p^{\prime}(u), u\right\rangle=0\right\}.
\]

Now, we have the following important lemma which is useful below.
\begin{lemma}\label{Lemasy}On lattice graph $\mathbb{Z}^N$,
\begin{enumerate}[label=(\roman*)]
\item If $V \not\equiv V_p$, then $c_{\mathcal{N}}<c_{p}$, where $c_{p}:=\inf _{u \in \mathcal{N}_{p}} J_{p}(u)$.
\item If $J\left(u_{n}\right) \rightarrow d \in\left(0, c_{p}\right)$ and $J^{\prime}\left(u_{n}\right) \rightarrow 0$, then $u_{n} \rightharpoonup u \neq 0$ after passing to a subsequence, $u$ is $a$ critical point of $J$ and $J(u) \leq d$.
\end{enumerate}
\end{lemma}
\begin{proof}
(i) Since $V \not\equiv V_p$ and $V_p(x) \geq V(x)$ for all $x \in \mathbb{Z}^N$, there exists $x_0 \in \mathbb{Z}^N$ such that $V_p(x_0) > V(x_0)$. It follows from Theorem \ref{theo120} that there exists a strictly positive (or negative) ground state solution $u_p$ for periodic problem \eqref{equp} such that $J_p(u_p) = c_p$.
Without loss of generality, here we can assume $u_p(x) > 0$ for all $x \in \mathbb{Z}^N$.
Let $s_p > 0$ be such that $s_pu_p \in \mathcal{N}$, then we obtain
\[
     J_p(s_pu_p) -  J(s_pu_p) \geq \frac{s_p^2}{2}\left(V_p(x_0) - V(x_0)\right)\abs{u_p(x_0)}^2 > 0.
\]
Thus,
 \[
        c_p = J_p(u_p) \geq J_p(s_pu_p) >  J(s_pu_p) \geq c_{\mathcal{N}}.
\]

(ii)By Lemma \ref{bounded}, $\{u_n\}$ is a bounded $(PS)$ sequence in  $X$ with the level $d$. Using Lemma \ref{lemembed}, we obtain $u_{n} \rightharpoonup u$ in $X, u_{n}(x) \rightarrow u(x)$ for all $x \in \mathbb{Z}^N$ after passing to a subsequence and, according to Lemma \ref{sub1} (vi), $u$ is a critical point of $J$. By Fatou's lemma,
$$
\begin{aligned}
d & =J\left(u_n\right)-\frac{1}{2}\left\langle J^{\prime}\left(u_n\right), u_n\right\rangle+o(1)=\frac{1}{2} \int_{\mathbb{Z}^N} u_n^2 d \mu+o(1) \\
& \geq \frac{1}{2} \int_{\mathbb{Z}^N} u^2 d \mu+o(1)=J(u)-\frac{1}{2}\left\langle J^{\prime}(u), u\right\rangle+o(1)=J(u)+o(1) .
\end{aligned}
$$
So $J(u) \leq d$ and it remains to show that $u \neq 0$. Argue by contradiction, suppose $u=0$. Since $u_n(x) \rightarrow 0$ for all $x \in \mathbb{Z}^N$ and $V(x) \rightarrow V_{p}(x)$ as $|x| \rightarrow \infty$,
$$
J\left(u_n\right)-J_{p}\left(u_n\right)=\frac{1}{2} \int_{\mathbb{Z}^N}\left(V(x)-V_{p}(x)\right) u_n^2 d \mu \rightarrow 0
$$
and therefore $J_{p}\left(u_n\right) \rightarrow d$. Using Lemma \ref{equi1} and the H\"{o}lder inequality and taking $w$ with $\|w\|=1$, we obtain
$$
\begin{aligned}
\left|\left\langle J^{\prime}\left(u_n\right)-J_{p}^{\prime}\left(u_n\right), w\right\rangle\right| & \leq \int_{\mathbb{Z}^N}\left(V_{p}(x)-V(x)\right)\left|u_n\right||w| d \mu \\
& \leq C\left(\int_{\mathbb{Z}^N}\left(V_{p}(x)-V(x)\right) u_n^2 d \mu\right)^{1 / 2} .
\end{aligned}
$$
As the right-hand side tends to 0 uniformly in $\|w\|=1, J^{\prime}\left(u_n\right)-J_{p}^{\prime}\left(u_n\right) \rightarrow 0$ and hence $J_{p}^{\prime}\left(u_n\right) \rightarrow 0$. Argue as the proof of Theorem \ref{theo120}, we obtain a sequence of $\left\{y_n\right\} \subset \mathbb{Z}^N$ and $\eta>0$ such that
$$
\vert u_n(y_n)\vert\geq \frac{\eta}{2} > 0
$$
 for all $n\in \mathbb{N}$. Since $u(x) = 0$ for all $x \in \mathbb{Z}^N$,  we have $\abs{y_n} \to \infty$. Let $v_{n}(x)=u_n\left(x+k_n T\right)$ with $k_n=\left(k_n^1, \ldots, k_n^N\right)$ to ensure that $\left\{y_n-k_n T\right\} \subset \Omega$, where  $\Omega := [0, T)^N \cap \mathbb{Z}^N$ is a bounded domain in $\mathbb{Z}^N$. Then, since $\norm{v_n}_2 = \norm{u_n}_2$ for all $n \in \mathbb{N}$, $\{v_n\}$ is bounded in $X$ by Lemma \ref{equi1}. Moreover, using Lemma \ref{lemembed}, we obtain $v_{n} \rightharpoonup v \neq 0$ in $X$ and $v_{n}(x) \to v(x)$ for all $x \in \mathbb{Z}^N$ after passing to a subsequence. Since $V_p(x)$ is $T$-periodic in $x$, $J_p$ and $J_p^{\prime}$ are invariant under the translation.
 Thus, we have
$$
J_p\left(v_n\right) \rightarrow d,\,\, J_p^{\prime}\left(v_n\right) \rightarrow 0.
$$
From Lemma \ref{sub1}(vi), we know that $v$ is a nontrivial critical point of $J_{p}$, and thus $J_{p}(v) \leq d< c_p$ which is a contradiction.
\end{proof}
\begin{proof}[Proof of Theorem \ref{theo130}]
If $V \equiv V_p$, then we obtain a ground state solution $u_p$ by Theorem \ref{theo120}. Therefore, without loss of generality, here we assume  $V(x_0) < V_p(x_0)$ for
some $x_0 \in \mathbb{Z}^N$.

From Lemma \ref{bounded} and Lemma \ref{mp}, there exists a bounded $(PS)$ sequence $\left\{u_n\right\} \subset X$ for $J$ with the level $c \in (0, c_p)$. By Lemma \ref{Lemasy}(ii), we obtain a critical point $u \neq 0$ of $J$ such that $J(u) \leq c$. So $u \in \mathcal{N}$. Since $c \leq c_{\mathcal{N}}$, it is easy to see that $c=c_{\mathcal{N}}$ and $u$ is a ground state solution of problem \eqref{loga}.

Moreover, arguing as in the proof of Theorem \ref{theo120}, we know $u(x)>0$ for all $x\in \mathbb{Z}^N$ or $u(x)<0$ for all $x\in \mathbb{Z}^N$.
\end{proof}
Next, we can regard the bounded potential as a special case of asymptotically periodic potential on lattice graph $\mathbb{Z}^N$.
Let $V$ be a bounded potential, i.e.,
$$
-1 \leq \inf _{x \in \mathbb{Z}^N} V(x) \leq \sup _{x \in \mathbb{Z}^N} V(x)=V_{\infty}<\infty \quad \text { with } \quad V_{\infty}=\lim _{|x| \rightarrow \infty} V(x).
$$
It is clear that $V_\infty$ is $1$-periodic, since $V_\infty\left(x+ e_i\right)=V_\infty, \forall x \in \mathbb{Z}^N, 1 \leq$ $i \leq N$, where $e_i$ is the unit vector in the $\mathrm{i}$-th coordinate.
\begin{proof}[Proof of Theorem \ref{theo140}]
Since $V_\infty$ is $1$-periodic,
$$
-1< \inf _{x\in\mathbb{Z}^N} V(x) \leq V(x) \leq V_\infty, \,\,\text{for all}\,\, x \in \mathbb{Z}^N,
$$
and
$$
\left\vert V(x)-V_\infty \right\vert \rightarrow 0 \,\,\text { as } \,\,|x| \rightarrow+\infty,
$$
Therefore, by considering  $V_\infty$ as $V_{p}$ in Theorem \ref{theo130},  we conclude the proof.
\end{proof}
Finally, we extend our results to a quasi-transitive graph $G$.
\begin{proof}[Proof of Theorem \ref{theo150}]
We modify our proofs of Theorem \ref{theo120}, \ref{theo130} and \ref{theo140}. Let $G/\Gamma  = \{\rho_1,...,\rho_m\}$ be the set of finitely many orbits, $\Omega = \{z_1,...,z_m\} \subset \mathbb{V}$ where $z_i \in \rho_i$, $1 \leq i \leq m$. We replace the translations in our proofs by $\Gamma$-action on functions respectively. For example, we replace $\{x_n - y_n\}$ with $\{g_n\} \subset \Gamma$ such that $\{g_n(x_n)\}\in \Omega$ and $\{x_n + y_n\}$ with $\{g^{-1}_n(x_n)\}$.
\end{proof}

\section{Extension to the $p$-Laplacian}
The graph $p$-Laplacian of $u \in C(\mathbb{V})$ is defined by
$$
\Delta_p u(x) := \sum_{y\sim{x}}\abs{u(y)-u(x)}^{p-2}(u(y)-u(x)).
$$
Let
\[
|\nabla u(x)|_p:=\left( \frac{1}{2} \sum_{y \in \mathbb{V}, y \sim x}\abs{u(y)-u(x)}^p\right)^{1 / p}
\]
and $W^{1,p}(\mathbb{V})$ be the completion of $C_c(\mathbb{V})$ under the norm
$$
\|u\|:= \left( \sum_{x \in \mathbb{V}}\abs{\nabla u(x)}_p^p + \sum_{x \in \mathbb{V}} \abs{u(x)}^p\right)^{1 / p} = \left(\frac{1}{2} \sum_{x \in \mathbb{V}} \sum_{y \sim x}\abs{u(y)-u(x)}^p+\sum_{x \in \mathbb{V}} \abs{u(x)}^p\right)^{1 / p}.
$$
For any function $u,v \in C(\mathbb{V})$, define
\[
\Gamma_p(u,v)(x) := \frac{1}{2} \sum_{y\sim{x}}\abs{u(y)-u(x)}^{p-2}(u(y)-u(x))(v(y)-v(x)).
\]

In this section, we intend to extend the previous results to the following $p$-Laplacian equation
\begin{equation}\label{pequation}
-\Delta_p u(x)+V(x)|u|^{p-2} u=|u|^{p-2} u \log |u|^p, \quad x \in \mathbb{V},
\end{equation}
where $p > 1$. The energy functional corresponding to \eqref{pequation} is
$$
J(u):=\frac{1}{p} \int_{\mathbb{V}}\left(|\nabla u|_p^p+(V(x)+1)|u|^p\right) d\mu-\frac{1}{p} \int_{\mathbb{V}}|u|^p \log |u|^p d\mu,
$$
If $G=(\mathbb{V},\mathbb{E})$ is the lattice graph $\mathbb{Z}^N$, then $J(u)$ is not of class $C^{1}$ in $W^{1,p}(\mathbb{Z}^N)$. One can find a function $u\in W^{1,p}(\mathbb{Z}^N)$ such that $\int_{\mathbb{V}}|u|^p \log |u|^p d\mu = - \infty$ in the appendix.

Since the argument on \eqref{eqlambda} holds for \eqref{pequation}, we still assume that $\underset{x\in \mathbb{Z}^N}\inf V(x)> -1$ in the rest of this section.

Our work space is as follows
\begin{equation}\label{eqnorm}
    X:=\left\{u \in W^{1,p}\left(\mathbb{V}\right): \int_{\mathbb{V}}\left(|\nabla u|_p^{p}+(V(x)+1) \abs{u}^{p}\right) d\mu<\infty\right\}
\end{equation}
with norm
$$
\norm{u}:=\left(\int_{\mathbb{V}}\left(|\nabla u|_p^{p}+(V(x)+1) \abs{u}^{p}\right) d\mu\right)^\frac{1}{p}.
$$

Similar to Theorem \ref{theo110}, we have the following result.
\begin{theorem}\label{theo111}
If $\lim _{d(x,x_0) \rightarrow \infty} V(x) = +\infty$ for some $x_0 \in \mathbb{V}$, then equation \eqref{pequation} has infinitely many solutions $\pm u_n$ such that $\lim _{n \rightarrow \infty} J\left( \pm u_n\right)=\infty$.
\end{theorem}

If the potential $V$ is coercive, then $X$ is a proper subspace of $W^{1,p}\left(\mathbb{V}\right)$, and $X$ is compactly embedded in $\ell^{q}\left(\mathbb{V}\right)$ for all $p \leq q \leq \infty$, see \cite[Lemma 2.1]{Shao}.

To prove Theorem \ref{theo111}, we need to prove a $p$-logarithmic Sobolev inequality on locally finite graphs.
\begin{lemma}\label{logaineq1}
For all $u \in \ell^p(\mathbb{V})$ such that $\norm{u}_p = 1$, there holds
    \begin{equation}
    \label{eqlogineqp1}
        \int_{\mathbb{V}}{\abs{u}^p\log{\abs{u}^p}}\, d\mu \leq 0.
    \end{equation}
\end{lemma}
\begin{proof}
Let  $u \in \ell^p(\mathbb{V})$ such that $\norm{u}_p = 1$. Since $\abs{u(x)}^p \leq \norm{u}_p^p = 1$ for all $x \in \mathbb{V}$, we have $\abs{u(x)}^p\log{\abs{u(x)}^p} \leq 0$ for any $x \in \mathbb{V}$ and $u(x)\neq 0$.  Thus
\[
    \int_{\mathbb{V}}{\abs{u}^p\log{\abs{u}^p}}\, d\mu \leq 0.
\]
\end{proof}
It follows from \eqref{eqlogineqp1} that, for any $u \in \ell^p(\mathbb{V})$ with $\norm{u}_p \neq 0$, we have
\begin{equation}
        \label{eqlogineqp}
        \int_{\mathbb{V}}\abs{u}^p\log{\abs{u}^p}\, d\mu \leqslant \norm{u}_p^p\log{\norm{u}_p^p}.
\end{equation}

The results of Section \ref{secframe}, including the equivalence of $W^{1,p}(\mathbb{V})$ and $\ell^p(\mathbb{V})$, the definition of the functionals $\Phi$ and $\Psi$, remain valid after making obvious changes (in particular, in the definition of $F_{1},-\frac{1}{2} s^{2} \log s^{2}$ should be replaced by $-\frac{1}{p}|s|^{p} \log |s|^{p}$ for $|s|<\delta$ ). In the proof of the boundedness part of Proposition \ref{propcompactps} inequality \eqref{eqlogineq2} must be replaced by \eqref{eqlogineqp}.
In the convergence part of the proof of Proposition \ref{propcompactps}, a formula corresponding to \eqref{eqps3} is
$$
\begin{aligned}
& \int_{\mathbb{V}}\left(\Gamma_p(u,u-u_n) + (V(x)+1)\left|u_{n}\right|^{p-2} u_{n}\left(u-u_{n}\right)\right) d\mu \\
& -\int_{\mathbb{V}} F_{2}^{\prime}\left(u_{n}\right)\left(u-u_{n}\right) d\mu + \Psi(u)-\Psi\left(u_{n}\right) \geq-\tau_{n}\left\|u-u_{n}\right\|,
\end{aligned}
$$
and this gives \eqref{eqpscon} with exponent $2$ replaced by $p$. In Lemma \ref{lemempty1} and Proposition \ref{propdefor1} only minor changes are needed.

On the other hand, the construction of $\left\{X_{k}\right\}$ and $\left\{Z_{k}\right\}$ requires great attention. Since $W^{1,p}(\mathbb{V})$ and $\ell^p(\mathbb{V})$ are
equivalent, we have the following lemma.
\begin{lemma}\label{lembasis}
$\{ \delta_x \}_{x \in \mathbb{V}}$ constitutes a Schauder basis of $X$.
\end{lemma}
 Let the Schauder basis of $X$ be $\left(e_{k}\right)_{k=1}^{\infty}$. Then ,there exists a set of biorthogonal functionals $\left(e_{k}^{*}\right)_{k=1}^{\infty} \subset X^{*}$, i.e. functionals such that $\left\langle e_{m}^{*}, e_{k}\right\rangle=\delta_{k m}$(cf.\cite[Section 1.b]{LT}). It is then easy to see that $e_{k}^{*}$ are total, i.e. $\left\langle e_{k}^{*}, u\right\rangle=0$ for all $k$ implies $u=0$(cf.\cite[Section 1.f]{LT}). Now we take $X_{k}:=\operatorname{span}\left\{e_{1}, \ldots, e_{k}\right\}$ and $Z_{k}:=\mathrm{cl} \operatorname{span}\left\{e_{k+1}, e_{k+2}, \ldots\right\}$, where cl denotes the closure, and we define $a_{k}, b_{k}, B_{k}, N_{k}$ as previously. Lemma \ref{lemgamma} still holds, essentially with the same proof as in \cite{W}, see \cite[Proposition 4.6]{SW}. In the proof of Proposition \ref{lemgamma} the exponent $2$ should be replaced by $p$ and $\beta_{k}^{p}\|u\|^{p}$ by $\beta_{k}^{q}\|u\|^{q}$, where $q \in \left(p,\infty\right)$. That $\beta_{k} \rightarrow 0$ follows from the fact that if $u_{k} \in Z_{k-1}$ and $u_{k} \rightharpoonup u$, then $u$ must be 0 because $e_{k}^{*}$ are total. The remaining part of the proof is unchanged.

For the equation \eqref{pequation} on $\mathbb{Z}^N$,  similar to Theorems \ref{theo120}, \ref{theo130} and \ref{theo140}, we have the following results.

\begin{theorem}\label{theo12}
Assume that $G=(\mathbb{V},\mathbb{E})$ be the lattice graph $\mathbb{Z}^N$. If $V(x)$  is  $T$-periodic in $x$, then equation \eqref{pequation} has a ground state solution $u$. Moreover, $u(x) > 0$ for all $x \in \mathbb{Z}^N$ or $u(x) < 0$ for all $x \in \mathbb{Z}^N$.
\end{theorem}
\begin{theorem}\label{theo13}
Assume that $G=(\mathbb{V},\mathbb{E})$ be the lattice graph $\mathbb{Z}^N$. If the potential $V(x)$ satisfies the assumptions:
\begin{enumerate}[label=($V_3$), ref=$V_\arabic{*}$]
\item\label{V3}
There exists a $T$-periodic function $V_p$ such that
$$
-1< \inf _{x\in\mathbb{Z}^N} V(x) \leq V(x) \leq V_p(x), \,\,\text{for all}\,\, x \in \mathbb{Z}^N,
$$
and
$$
\left|V(x)-V_p(x)\right| \rightarrow 0 \quad \text { as } \quad|x| \rightarrow+\infty.
$$
\end{enumerate}
then problem \eqref{pequation} has a ground state solution $u$. Moreover, $u(x) > 0$ for all $x \in \mathbb{Z}^N$ or $u(x) < 0$ for all $x \in \mathbb{Z}^N$.
\end{theorem}
\begin{theorem}\label{theo14}
Assume that $G=(\mathbb{V},\mathbb{E})$ be the lattice graph $\mathbb{Z}^N$. If the potential $V$ satisfies the following assumption
 \begin{enumerate}[label=($V_4$), ref=$V_\arabic{*}$]
\item\label{V4}
$$
-1 \leq \inf _{x \in \mathbb{Z}^N} V(x) \leq \sup _{x \in \mathbb{Z}^N} V(x)=V_{\infty}<\infty \,\, \text { with } \,\, V_{\infty}=\lim _{|x| \rightarrow \infty} V(x)
$$
\end{enumerate}
then problem \eqref{pequation} has a ground state solution $u$. Moreover, $u(x) > 0$ for all $x \in \mathbb{Z}^N$ or $u(x) < 0$ for all $x \in \mathbb{Z}^N$.
\end{theorem}
In Theorem \ref{theo12}, \ref{theo13} and \ref{theo14}, we work in the space $X=W^{1, p}\left(\mathbb{V}\right)$ with the same norm as in \eqref{eqnorm}. The essence of the proofs is basically the same as in Theorem \ref{theo120}, \ref{theo130} and \ref{theo140}, so we omit them here.

Similar to Theorem \ref{theoquasi}, we also can extend our results to a quasi-transitive graph  $G$.
\begin{theorem}
\label{theoquasip}
Let $G = (\mathbb{V},\mathbb{E})$ be a quasi-transitive graphs. Suppose $\Gamma \leq Aut(G)$ and the
action of $\Gamma$ on $G$ has finitely many orbits . Then, by replacing $T$-periodicity of $V$ with $\Gamma$-invariance, Theorem \ref{theo12} holds on $G$. By replacing $\abs{x}$ with $d(x,x_0)$ for some $x_0 \in \mathbb{V}$, Theorem \ref{theo13} and \ref{theo14} holds on $G$.
\end{theorem}

\subsection*{Acknowledgements}
C. Ji was partially supported by National Natural Science Foundation of China (No. 12171152).

\section*{Appendix}
\begin{example}
Let $p>1$, using the fact
\begin{equation}\label{EXA}
\sum_{n=2}^{\infty} \frac{1}{n(\log n)^q} \begin{cases}<\infty, & \text { if } q > 1, \\ =\infty, & \text { if } 0 \leq q \leq 1,\end{cases}
\end{equation}
we can show that there exists $u \in W^{1,p}(\mathbb{Z}^N)$
with $\int_{\mathbb{Z}^N} \abs{u}^p \log \abs{u}^p d \mu = - \infty$.\\

Let
\[
        e := (1,0,...,0) \in \mathbb{Z}^N,
\]
and
\[
    u(x)= \begin{cases}|x|^{-\frac{1}{p}} (\log |x|)^{-\frac{2}{p}}, & x = ne, n \geqslant 3, \\ 0, & else.\end{cases}
\]
Then we have
\[
    \int_{\mathbb{Z}^N} \abs{u}^p d \mu=\sum_{x \in \mathbb{Z}^N} \abs{u(x)}^p = \sum_{n = 3}^\infty \frac{1}{n(\log n)^2}<\infty.
\]
Thus $u \in \ell^p(\mathbb{Z}^N)$. Similar to Lemma \ref{equi1}, we have $u \in W^{1,p}(\mathbb{Z}^N)$. However,
\[
\begin{aligned}
    \int_{\mathbb{Z}^N} \abs{u}^p \log \abs{u}^p d \mu & =\sum_{x \in \mathbb{Z}^N} \abs{u(x)}^p \log \abs{u(x)}^p \\
    & =-\left(\sum_{n = 3}^\infty \frac{1}{n \log n} + \sum_{n = 3}^\infty \frac{2 \log (\log n)}{n(\log n)^2}\right) \\
    & =-\infty
\end{aligned}
\]
where we use the fact \eqref{EXA} and $\frac{2 \log (\log n)}{n(\log n)^2}>0$ for $n \geq 3$.
\end{example}

  \end{document}